\documentclass[12pt]{article}
\usepackage{amsmath,amssymb,amsthm, blindtext, mathtools}
\usepackage[mathscr]{eucal}
\usepackage[usenames,dvipsnames]{xcolor}
\usepackage{enumerate}

\DeclareMathOperator{\dist}{dist}
\newcommand{\bndry}{b}

\newcommand{\CintD}{\mathbf C_{\domain}}

\newcommand{\CopD}{\mathcal C_{\domain}}
\newcommand{\domain}{D}

\newcommand{\dee}{\partial}
\newcommand{\deebar}{\overline\dee}
\newcommand{\ha}{h}
\newcommand{\neu}{\mathfrak{n}}

\newcommand{\Hs}{\mathcal H}
\newcommand{\Hsa}{\mathcal H_\alpha}

\newcommand{\Log}{\mathrm{Log}}

\newcommand{\abs}[1]{\left\vert#1\right\vert}

\long\def\symbolfootnote[#1]#2{\begingroup
\def\thefootnote{\fnsymbol{footnote}}\footnote[#1]{#2}\endgroup}

\newtheorem{thm}{Theorem}[section]
\newtheorem{prop}[thm]{Proposition}
\newtheorem{lem}[thm]{Lemma}

\theoremstyle{definition}

\newtheorem{defn}[thm]{Definition}

\theoremstyle{remark}

\newtheorem{rem}[thm]{Remark}

\title{New Properties of Holomorphic Sobolev-Hardy Spaces}
\author{William Gryc, Loredana Lanzani\footnote{supported in part by the National Science Foundation, award no. DMS-1901978, and a Simons Foundation Travel Support for Mathematicians, award no. 919763.}, Jue Xiong, Yuan Zhang}
\renewcommand{\thefootnote}{\fnsymbol{footnote}} 
\footnotetext{\emph{2020 Mathematics Subject Classification.} 30H10, 30E20, 30E25, 31A25.}
\footnotetext{\emph{Key words.} Lipschitz domain,  Sobolev-Hardy space, Neumann   problem, $\bar\partial$, reproducing kernel Hilbert space.}     
\date{}
\begin{document}

\maketitle

\begin{center}\textit{Dedicated to Steven G. Krantz}\end{center}
\vspace{0.1in}

\begin{abstract}
We give new characterizations of the optimal data space for the $L^p(\bndry\domain,\sigma)$-Neumann boundary value problem for the $\bar{\partial}$ operator associated to a bounded, Lipschitz domain $\domain\subset\mathbb{C}$. We show that the solution space is embedded (as a Banach space) in the Dirichlet space and that for $p=2$, the solution space is a reproducing kernel Hilbert space.
\end{abstract}

\section{Introduction}
Let $\domain$ be a bounded Lipschitz domain in $\mathbb C$ whose boundary $\bndry\domain$ is endowed with the induced Lebesgue measure $\sigma$. Let $\Hs^p(\domain)$ be the {\bf holomorphic Hardy space}:
\begin{equation*}
\Hs^p(\domain):= \{F\in\vartheta (D):\, F^*\in L^p(\bndry D, \sigma)\},\ \quad 0<p\leq \infty 
\end{equation*}
  with $\vartheta(\domain)$  denoting the set of holomorphic functions on $\domain$ and $F^*$   the non-tangential maximal function of $F$. It is well-known that if $\domain$ is simply connected, every element $F$ of $\Hs^p(\domain)$ admits a nontangential limit $\dot{F}$  that lies in $L^p(\bndry\domain,\sigma)$ (see \cite[Theorem 10.3]{Duren}). On the other hand, since Lipschitz domains are local epigraphs, any bounded Lipschitz domain must be finitely connected. Hence, an elementary localization argument shows that any $F\in\Hs^p(\domain)$ has a nontangential limit $\dot{F}$ defined $\sigma$-a.e. on $\bndry\domain$. We will call the set of all such nontangential limits $h^p(\bndry\domain)$. That is,
\begin{equation*}
h^p(\bndry\domain) := \left\{\dot F\ :\ F\in \Hs^p(\domain)\right\}\, \subsetneq\, L^p(\bndry\domain, \sigma).
\end{equation*}

Let $ \Hs^{1, p}(\domain)$ be the {\bf holomorphic Sobolev-Hardy space}
    \begin{equation*}
    \Hs^{1, p}(\domain): = \{G\in \mathcal \vartheta(\domain): G'\in \Hs^p(\domain)\},\quad p>0\, .
\end{equation*}
It is shown in \cite{GLZ} that, given $g\in L^p(\bndry\domain, \sigma)$ subject to the compatibility condition: 
$\displaystyle{\int\limits_{\bndry\domain}g\, d\sigma =0}$,  the  {\bf Neumann problem for the $\deebar$ operator}
\begin{equation}\label{NdbarIntro}
     \left\{
      \begin{array}{lcll}
      \bar\partial G &= &0 & \text{in} \ \ \domain;\\ \\
       \displaystyle{\frac{\partial G}{\partial n}(\zeta)} &= &g(\zeta) & \text{for}\ \sigma\text{-a.e.}\ \zeta\in \bndry\domain;
       \\ \\
      (G')^*&\in& L^p(\bndry\domain, \sigma)
      \end{array}
\right.
\end{equation}
 is solvable if and only if the data $g$ belongs to 
\begin{equation}\label{E:DefNeumannData}
\neu^{p}(\bndry\domain):= \left\{-iT(\zeta)\dot{(G')}(\zeta): \ G\in  \mathcal H^{1, p}(\domain)\right\},\quad 1\leq p\leq\infty,
\end{equation}
where $\zeta\mapsto T(\zeta)$ is the unit tangent vector field for $\bndry\domain$. Moreover, if
$g\in \neu^{p}(\bndry\domain)$ then 
all solutions  of \eqref{NdbarIntro} belong to $ \Hs^{1, p}(\domain)$. 
 Any two solutions of \eqref{NdbarIntro} differ by an additive constant, hence for any fixed $\alpha\in\domain$ the space
\[\Hsa^{1, p}(\domain) := \{F\in \Hs^{1,p}(\domain): F(\alpha)=0\}\] 
contains precisely one solution of \eqref{NdbarIntro}.
In the case when $p=2$ and $\domain$ is simply-connected, $\Hsa^{1, 2}(\domain)$ is a Hilbert space with inner product
\begin{equation*}
\langle F, G\rangle_{\Hsa^{1, 2}(\domain)} := \int\limits_{\bndry\domain}\!\! \dot{(F')}(\zeta)
\,\overline{\dot{(G')}(\zeta)}\,d\sigma(\zeta).
\end{equation*} 

In this paper we explore properties of $\Hs_\alpha^{1, 2}(\bndry\domain)$ and of $\neu^{p}(\bndry\domain)$.
Specifically, after recalling a few well-known basic properties of Lipschitz domains (Section \ref{section2}),  
we show that the solution space $\Hsa^{1,2}(\domain)$ is a reproducing kernel Hilbert space (Theorem \ref{HilbertSpaceProp}) and for $D=\mathbb{D}$ (the unit disc) we compute its reproducing kernel. Next we show that for $1<p<\infty$ there is a Banach space embedding of $\Hsa^{1,p}(\domain)$ in the Dirichlet space $\mathcal{D}^p_\alpha(\domain)$ (Theorem \ref{DirichletProp}). In Section \ref{section3} we give various characterizations of $\neu^p(\bndry\domain)$ for simply connected $\domain$: in terms of $L^p(\bndry\domain,\sigma)$-functions whose moments all vanish on $\bndry\domain$; or in terms of the vanishing of the Cauchy integral over $\overline{\domain}^c$, the complement of the closure of $\domain$; as well as in terms of its conformal map  (Theorem \ref{h^{1,p}Remarka} and Theorem \ref{PropVanish}). 
 Finally, in Section \ref{sectionmc} we provide a characterization of $\neu^p(\bndry\domain)$ for
  multiply connected $\domain$: in this case the aforementioned vanishing moment condition takes a more restrictive form,  see Theorem \ref{Nm:SHsp-bd}.
\vskip0.1in
\noindent\textbf{Acknowledgement: } This work was started at the AIM workshop {\em Problems on Holomorphic Function Spaces \& Complex Dynamics}, an activity of the AWM Research Network in Several Complex Variables.
We are grateful to the American Institute of Mathematics and  the Association for Women in Mathematics for their hospitality and support. 
%%%%%%%%%%%%%%%%%%%%%%%%%%%%%%%%%%%%%%%%%%%%%%%%%%%%%%%%%%%%%%%%%%%%%%%%%%%%%%%%%%%%%%%%%%%%%%%%%%%%%%%%%%%%%%%%%%%%%%%%%%%%%%%%%%%%%%%%%%%%%%%%%%%%%%%%%%%%%%%%
\section{Preliminaries}\label{section2}
\subsection{Lipschitz domains}
Throughout this paper the domains under consideration will be Lipschitz domains on $\mathbb C$, as defined below.
\begin{defn}\label{de}
A bounded domain $\domain\subset\mathbb{C}$ with boundary $\bndry\domain$ is called a \textbf{Lipschitz domain} if there are finitely many rectangles $\{R_j\}_{j=1}^m$ with sides parallel to the coordinate axes, angles $\{\theta_j\}_{j=1}^m$, and Lipschitz functions $\phi_j:\mathbb{R}\to\mathbb{R}$ such that the collection $\{e^{-i\theta_j}R_j\}_{j=1}^m$ covers $\bndry\domain$ and $(e^{i\theta_j}\domain)\cap R_j=\{x+iy: y > \phi_j(x),\  x\in (a_j, b_j)\}$ for some $a_j<b_j<\infty$. We refer to such $R_j$'s as {\em coordinate rectangles}.
\end{defn}
\begin{defn}\label{D:ap}
Let $\domain$ be a bounded Lipschitz domain. For any $\zeta\in\bndry \domain$, let  
$\{\Gamma(\zeta), \zeta\in \domain\}$ be a family of truncated (one-sided) open 
 cones $\Gamma(\zeta)$ with vertex at $\zeta$ satisfying the following property: for each rectangle $R_j$ in Definition \ref{de}, there exists two cones $\Delta_1$ and $\Delta_2$, each with vertex at the origin and axis along the $y$ axis such that for $\zeta\in \bndry\domain \cap e^{-i\theta_j}R_j$,
$$ e^{-i\theta_j}\Delta_1 +\zeta\quad \subset\quad  \Gamma(\zeta)\quad \subset\quad  \overline{\Gamma(\zeta)}\setminus \{\zeta\}\quad \subset\quad  e^{-i\theta_j}\Delta_2+\zeta \quad  \subset\quad   \domain\ \cap\  e^{-i\theta_j}R_j. $$
\end{defn}
It is well known that for Lipschitz $\domain$, $\Gamma (\zeta)\neq \emptyset$ for any $\zeta \in \bndry\domain$; see e.g.,  \cite{Dah2} or \cite[Section 0.4]{V}. We will sometimes refer to $\Gamma (\zeta)$ as a {\em regular cone}, or a {\em coordinate cone}. For a function $F$ on $\domain$ and $\zeta\in\bndry\domain$, we define the \textbf{nontangential maximal function} $F^*(\zeta)$ and the \textbf{nontangential limit} $\dot{F}(\zeta)$ as

$$
F^*(\zeta):=\sup\limits_{z\in\Gamma(\zeta)}|F(z)|\, , \qquad \text{and}\quad \dot{F}(\zeta)=
\lim_{\stackrel{z\to \zeta}{ z\in \Gamma (\zeta)}}
F(z)\quad \text{if such limit exists.}
$$

We will need  an approximation scheme of $\domain$ by smooth subdomains constructed by Ne\v{c}as in \cite{Necas}, which we refer to as a \textbf{Ne\v{c}as exhaustion of $\domain$}. See also \cite{L1} and \cite[Theorem 1.12]{V}. (Recall that Lipschitz functions are differentiable almost everywhere; thus if $\domain$ is Lipschitz and simply connected its boundary $\bndry\domain$ is a rectifiable Jordan curve  that admits a (positively oriented) unit tangent vector $T(\zeta)$ as $\sigma$-a.e. $\zeta\in\bndry\domain$.)
\begin{lem}\cite[p. 5]{Necas}\cite[Theorem 1.12]{V}\label{L:NecasExhaustion}
Let $\domain$ be a bounded Lipschitz domain. There exists a family $\{\domain_k\}_{k=1}^\infty$ of smooth domains with $\domain_k$ compactly contained in $\domain$ that satisfy the following:
\begin{enumerate}[(a).]
\item For each $k$ there exists a Lipschitz diffeomorphism $\Lambda_k$ that takes $\domain$ to $\domain_k$ and extends to the boundaries: $\Lambda_k:\bndry\domain\to \bndry\domain_k$ with the property that
\[\sup\{|\Lambda_k(\zeta)-\zeta|:\zeta\in \bndry\domain\}\leq C/k\]
for some fixed constant $C$. Moreover $\Lambda_k(\zeta)\in \Gamma(\zeta)$.
\item There is a covering of $\bndry\domain$ by finitely many coordinate rectangles which also form a family of coordinate rectangles for $\bndry\domain_k$ for each $k$. Furthermore for every such rectangle $R$, if $\phi$ and $\phi_k$
 denote the Lipschitz functions whose graphs describe the boundaries of $\domain$ and $\domain_k$, respectively, in $R$, then $\|(\phi_k)'\|_\infty \leq \|\phi'\|_\infty$ for any $k$; $\phi_k\to\phi$ uniformly as $k\to\infty$, and $(\phi_k)'\to\phi'$ a.e. and in every $L^p((a, b))$  with $(a, b)\subset\mathbb R$ as in Definition \ref{de}.
\item There exist constants $0<m<M<\infty$ and positive functions (Jacobians) $w_k:\bndry\domain\to [m,M]$ for any $k\in \mathbb N$, such that for any measurable set $F\subseteq \bndry\domain$ and for any measurable function $f_k$ on $\Lambda_k (F)$ the following change-of-variables formula holds:
\[\int\limits_{F}
\!\!
f_k(\Lambda_k(\eta))\,w_k(\zeta)\,d\sigma(\eta) = \int\limits_{ \Lambda_k(F)}
\!\!\!
f_k(\eta_k)\,d\sigma_k(\eta_k).\]
where $d\sigma_k$ denotes arc-length measure on $\bndry\domain_k$.  Furthermore we have
 $$w_k\to  1 \quad \sigma\text{-a.e.}\ \bndry\domain\ \  \text{and in every}\quad  L^p(\bndry\domain,\sigma)\ \ \text{for any}\ \  1\leq p <\infty\, .$$
\item Let $T_k$ denote the unit tangent vector for $\bndry\domain_k$ and $T$ denote the unit tangent vector of $\bndry\domain$. We have that

$$T_k\to  T \quad \sigma\text{-a.e.}\ \bndry\domain\ \  \text{and in every}\quad  L^p(\bndry\domain,\sigma)\ \ \text{for any}\ \  1\leq p <\infty\, .$$
\end{enumerate}
\end{lem}
\vskip0.1in

Note that in conclusions {\em (b)} through  {\em (d)} the exponent $p=\infty$ cannot be allowed unless $\domain$ is of class $C^1$.  Ne\v{c}as exhaustions can be used to transfer well-known results for holomorphic functions over domains with smooth boundaries to Hardy space functions on Lipschitz domains. In particular, one can use it to prove Cauchy's Theorem. See also \cite[Lemma 2.7]{GLZ} for the proof.

\begin{lem}\label{L:CauhyThmHp}
Let $\domain$ be a bounded Lipschitz domain. Then any $f\in h^1(\bndry\domain)$ satisfies Cauchy's Theorem. That is
\begin{equation*}
\int\limits_{\bndry\domain} f(\zeta)\, d\zeta = 0\qquad \text{for any}\quad  f\in h^1(\bndry\domain).
\end{equation*}
\end{lem}

 Next we state some definitions and results involving Cauchy integrals and the Cauchy transform, which we first define:
\begin{defn}
Let $f:\bndry\domain\to\mathbb{C}$. The \textbf{Cauchy integral} $\mathbf C_D f$ of $f$ is
\[\mathbf C_\domain f(z) := \frac{1}{2\pi i}\int\limits_{\bndry\domain}\frac{f(\zeta)}{\zeta - z}\, d\zeta,\quad z\in \domain.\]
Similarly
\[\mathbf C_{\overline{D}^c} f(z) := \frac{1}{2\pi i}\int\limits_{\bndry\domain}\frac{f(\zeta)}{\zeta - z}\, d\zeta,\quad z\in \overline{D}^c.\]
Finally, the \textbf{Cauchy transform} $\CopD f$ of $f$ is denoted by
\[\CopD f(\zeta) := \dot{(\CintD f)}(\zeta),\quad \zeta\in\bndry\domain.\]
In both integrals $\bndry\domain$ is oriented counterclockwise (that is, in the positive direction for $\domain$).
\end{defn}
In this paper we will use the fact that a function $f$ in $L^p(\bndry\domain,\sigma)$ lies in $\ha^p(\bndry\domain, \sigma) $ if and only if the Cauchy integral of $f$ vanishes on $\overline{\domain}^c$. This latter fact is well-known for domains with smooth boundaries; here we prove it for Lipschitz domains, see Lemma \ref{hpCharacterization} below. We first recall the Plemelj formulas for $f\in L^p(\bndry\domain,\sigma)$, $1<p< \infty$:
\begin{equation}\label{Plemelj1}
\CopD f(\zeta) = \frac{1}{2}f (\zeta)+ \frac{1}{2}\EuScript{HC}_{\bndry\domain}f(\zeta), \quad \mbox{ for $\sigma$-a.e. $\zeta\in\bndry\domain$},
\end{equation}
and
\begin{equation}\label{Plemelj2}
\lim_{\stackrel{z\to \zeta}{ z\in \Gamma (\zeta,\overline{ D}^c)}} \mathbf C_{\overline{D}^c} f(z) =  -\frac{1}{2}f(\zeta) + \frac{1}{2}\EuScript{HC}_{\bndry\domain}f(\zeta) \mbox{ for $\sigma$-a.e. $\zeta\in\bndry\domain$}.
\end{equation}
Here 
\begin{equation*}
\EuScript{HC}_{\bndry\domain} f(\zeta) := \frac{1}{2\pi i}\,\text{P.V.}\!\!\int\limits_{\bndry\domain}\frac{f(w)}{w - \zeta}\, dw,\quad \text{for}\ \sigma\text{-a.e.}\ \ \zeta\in \bndry\domain,
\end{equation*}
with $\bndry\domain$ oriented counterclockwise, and  $\Gamma (\zeta,\overline{ D}^c)$ is defined as in Definition \ref{D:ap},  with $D$ in there replaced by $\overline{ D}^c$. Note that a Lipschitz domain $\domain$  satisfies the exterior cone condition (see \cite{Kenig-2}) so the limit in \eqref{Plemelj2} is well-defined.
A deep result of Coifman, McIntosh, and Meyer \cite{CMM} states that on bounded Lipschitz domains,  $\EuScript{HC}_{\bndry\domain}$ is indeed well-defined (i.e. the principal value integral exists $\sigma$-a.e.) and  is bounded on $L^p(\bndry\domain,\sigma)$,  $1<p<\infty$. Thus, by the result of \cite{Bagemihl}, the Plemelj formulas \eqref{Plemelj1} and \eqref{Plemelj2} hold (for more on Plemelj formulas, also see \cite{Musk}).
\begin{lem}\label{hpCharacterization}
Let  $\domain$ be a bounded simply connected Lipschitz domain and $1<p<\infty$. Assume $f\in L^p(\bndry\domain,\sigma)$. Then $f\in\ha^p(\bndry\domain, \sigma) $ if and only if $\mathbf C_{\overline{\domain}^c} f(z) = 0$ for all $z\in\overline{\domain}^c$.
\end{lem}
\begin{proof}
First assume that $\mathbf C_{\overline{D}^c} f(z) = 0$ for all $z\in\overline{\domain}^c$. By Equation \eqref{Plemelj2}, we have
\[0=\lim_{\stackrel{z\to \omega}{ z\in \Gamma (\zeta,\overline D^c)}}
 C_{\overline{D}^c} f(z) = \frac{1}{2\pi i}\mathrm{P.V.}\int\limits_{\bndry\domain} \frac{f(\zeta)}{\zeta-\omega}d\zeta-\frac{1}{2}f(\omega).\]
That is, $\frac{1}{2}f(\omega)= \frac{1}{2\pi i}\mathrm{P.V.}\int\limits_{\bndry\domain} \frac{f(\zeta)}{\zeta-\omega}$ for $\sigma$-a.e. $\omega\in \bndry\domain$. Now, using Equation \eqref{Plemelj1}, we have for $\sigma$-a.e.  $\omega\in\bndry\domain$,
\[\CopD f(\omega) = \frac{1}{2\pi i}\mathrm{P.V.}\int\limits_{\bndry\domain} \frac{f(\zeta)}{\zeta-\omega}d\zeta+\frac{1}{2}f(\omega) = \frac{1}{2}f(\omega)+\frac{1}{2}f(\omega) = f(\omega).\]
Thus, $f$ is in the range of the Cauchy transform. Since the range of the Cauchy transform equals $\ha^p(\bndry\domain, \sigma)$ when $\domain$ is bounded and simply connected and $1<p<\infty$ (see \cite{L1}), the backward direction is proven. For the forward direction suppose $f\in\ha^p(\bndry\domain, \sigma) $. Then there exists $F\in\Hs^p(\domain)$ such that $\dot{F}=f$. Let $z\in\overline{D}^c$ be arbitrary and consider the function $G_z(w):=(w-z)^{-1}$. Then $G_z$ is holomorphic on $\domain$ and is continuous on $\overline{D}$. Moreover,  $\|(FG_z)^*\|_{L^p(\bndry\domain,\sigma)}\leq \|F^*\|_{L^p(\bndry\domain,\sigma)}\|G_z^*\|_{L^\infty(\bndry\domain,\sigma)}<\infty$. Thus $FG_z\in \Hs^p(\domain)$ and by Cauchy's Theorem (Lemma \ref{L:CauhyThmHp})
we have
\[0 = \int\limits_{\bndry\domain} \dot{(FG_z)}(\zeta) d\zeta = \mathbf C_{\overline{D}^c} f(z),\]
as desired.
\end{proof}
%%%%%%%%%%%%%%%%%%%%%%%%%%%%%%%%%%%%%%%%%%%%%%%%%%%%%%%%%%%%%%%%%%%%%%%%%%%%%%%%%%%%%%%%%%%%%%%%%%%%%%%%%%%%%%%%%%%%%%%%
\section{Properties of $\Hsa^{1,2} (\domain)$ for simply connected $\domain$}\label{section4}
In this section we show that $\Hsa^{1,2}(\domain)$ is a reproducing kernel Hilbert space and that it is a subset of the Dirichlet space. 
%%%%%%%%%%%%%%%%%%%%%%%%%%%%%%%%%%
\subsection{$\Hsa^{1, 2}(\domain)$ is a reproducing kernel Hilbert space}
%%%%%%%%%%%%%%%%%%%%%%%%%%%%%%%%%%%
\begin{thm}\label{HilbertSpaceProp}
Let $\domain$ be a bounded simply connected  Lipschitz domain. Then
for any base point $\alpha\in \domain$:
\begin{enumerate}[(a)]
\item\label{HS1} $\Hsa^{1, 2}(\domain)$ is a Hilbert space with inner product
\vskip0.05in
\centerline{$\displaystyle{
\langle F, G\rangle_{\Hsa^{1,2}(\domain)} := \langle \dot{(F')}, \dot{(G')}\rangle_{L^2(\bndry \domain, \sigma)}.
}$}
\vskip0.12in
\item \label{HS2} For any $z\in\domain$, the pointwise evaluation: $G\mapsto E_z(G) := G(z)$ is a bounded linear functional on $\Hsa^{1, 2}(\domain)$. Hence  $\Hsa^{1, 2}(\domain)$ is  a reproducing kernel Hilbert space (RKHS)  with reproducing kernel $K_\alpha^z(\cdot)=K_\alpha(\cdot,z)$. Namely, for any $z\in\domain$, we have that

$$\dot{(K^z_\alpha)'}(\zeta)\ \ \equiv \ \
 \lim_
{\stackrel{w\to\zeta}{w\in\Gamma(\zeta)}}
  (K^z_\alpha)'(w)$$

exists for almost all $\zeta\in\bndry\domain$ and for $F\in \Hsa^{1, 2}(\domain)$ we have
\begin{equation}\label{E:repr}
    F(z)=\int_{\bndry\domain}\! \dot {({F'})}(\zeta)\, \overline{(\dot{K_\alpha^z})'(\zeta)}\,d\sigma(\zeta),\ \ z\in \domain.
\end{equation}

\item  \label{HS3} Let  $p\geq 2$ and $g\in \neu^{p}(\bndry\domain)$. Then for any $\alpha\in\domain$ the solution of the holomorphic Neumann problem
\eqref{NdbarIntro} with boundary data $g$ has the representation
\vskip0.07in
\centerline{$\displaystyle{
G_\alpha(z) = i\int\limits_{\zeta\in\bndry D}\!\!\! g(\zeta)\,\overline{T(\zeta)\,\dot{(K^{z}_\alpha)'}(\zeta)}\, d\sigma(\zeta),\ \ z\in D.
}$}
\end{enumerate}
\end{thm}
\begin{proof}
To verify \textit{($\ref{HS1}$)}, note that $\langle \cdot, \cdot\rangle_{\Hsa^{1,2}(\domain)}$ is a sesquilinear form and $\langle F, F\rangle_{\Hsa^{1,2}(\domain)}=\|F\|_{\Hsa^{1,2}(\domain)}^2$. A straightforward argument (whose details can be found in \cite[Lemma 3.4]{GLZ}) shows that for $1\leq p\leq \infty$ the set $\mathcal H_\alpha^{1, p}(\domain)$ is a Banach space with the norm  defined as
 \[\|F\|_{\mathcal H_\alpha^{1, p}(\domain)}= \|\dot{(F')}\|_{L^p(\bndry\domain,\sigma)}.\]
Thus $\Hsa^{1,2}(\domain)$ is complete under the norm $\|\cdot\|_{\Hsa^{1,2}(\domain)}$, and so $\Hsa^{1,2}(\domain)$ is a Hilbert space.

Next we prove \textit{($\ref{HS2}$)}. Fix $z\in \domain$ and consider the pointwise evaluation operator $E_z$. For any $\alpha\in\domain$ and a smooth path $\gamma_\alpha^z\subset\domain$ that connects $\alpha$ to $z$ we have
\begin{equation*}
|E_z(G)|=|G(z)| = |G(z)-G(\alpha)|=\left|\,\int\limits_{\gamma_\alpha^z} G'(w)dw\right| \leq |\gamma_\alpha^z| \sup_{w\in\gamma_\alpha^z} |G'(w)| \, .
\end{equation*}
Furthermore, for any $w\in \gamma_\alpha^z$, Cauchy formula and H\"older inequality give
\begin{equation*}
|G'(w)|
=
\frac{1}{2\pi}
 \left|\,\int\limits_{\bndry\domain} \frac{\dot{(G')}(\zeta)}{w-\zeta}d\zeta\right|
\leq  \frac{
|\bndry\domain|^{\frac{1}{2}}}{2\pi k_z}\|\dot{(G')}\|_{L^2(\bndry\domain,\sigma)}=\frac{
|\bndry\domain|^{\frac{1}{2}}}{2\pi k_z}\|G\|_{\Hsa^{1,2}(\domain)},
\end{equation*}
where $k_z:= \dist (\gamma_\alpha^z,\bndry\domain) >0$. Combining all of the above we see that for any $z\in \domain$, $E_z$ is a bounded linear functional on $\Hsa^{1,2}(\domain)$; Hilbert space theory now grants the existence of the reproducing kernel function $$K^z_\alpha\in  \Hsa^{1,2}(\domain)\quad \text{with}\quad G(z)=\langle G, K^z_\alpha \rangle_{\Hsa^{1,2}(\domain).}$$

Finally we verify \textit{($\ref{HS3}$)}. Let $p\geq 2$ and $g\in \neu^{p}(\bndry\domain)$. Suppose $G_\alpha\in\Hsa^{1,p}(\domain)$ is the solution to the Neumann problem \eqref{NdbarIntro} with datum $g$. Thus $\dot{(G_\alpha')} =i\overline{T}g $ and $G_\alpha\in \Hsa^{1,2}(\domain)$. Hence for any $z\in \domain$ we have
\begin{equation*}
G_\alpha(z) = \langle G_\alpha, K^z_\alpha \rangle_{\Hsa^{1,2}(\domain)} = \int\limits_{\bndry\domain}\dot{(G'_\alpha)}(\zeta)\overline{(\dot{K^z_\alpha})'(\zeta)}d\sigma(\zeta)
= i\int\limits_{\bndry\domain}g(\zeta)\overline{T(\zeta)(\dot{K^z_\alpha})'(\zeta)}d\sigma(\zeta),
\end{equation*}
as desired.
\end{proof}

In the case of the unit disc $\mathbb{D}$ we obtain explicit formulas and recover the full range of $1\leq p\leq \infty$:
\begin{thm}\label{P:disc}
\begin{enumerate}
\item The reproducing kernel associated to $\Hsa^{1,2}(\mathbb D)$ is given by
\begin{equation}\label{ReproducingKernel}
K^z_\alpha(w)=\sum_{k=1}^{\infty}\frac{(w^k-\alpha^k)\overline {(z^k-\alpha^k)}}{2\pi k^2},\qquad z,w\in\mathbb{D}.
\end{equation}
\item Given $g\in \neu^{p}(b\mathbb D)$, $1\le p\le \infty$ and $\alpha:=0$, the unique solution  $G\in \Hs_0^{1,p}(\mathbb D)$ to the holomorphic Neumann problem \eqref{NdbarIntro}    admits the following representation
\begin{equation}\label{dr}
    G(z)=\frac{1}{2\pi}\int\limits_{b\mathbb D}g(\zeta)\Log\frac{1}{1-z\overline \zeta}\,d\sigma(\zeta),
\end{equation}
where $\Log$ denotes the principal branch of the complex logarithm.
\end{enumerate}
\end{thm}
\begin{proof}
To prove part \textit{1.}, note that since $\mathbb{D}$ is simply connected every holomorphic function on $\mathbb{D}$ has an antiderivative. Thus the mapping $G\mapsto G'$ is an isometric isomorphism from $\Hsa^{1,2}(\mathbb{D})$ onto $\Hs^{2}(\mathbb{D})$. Since $\{{\frac{1}{ \sqrt{2\pi}}} z^{k-1}\}_{k\in\mathbb{N}}$ is an orthonormal basis of $\Hs^{2}(\mathbb{D})$, the set of antiderivatives $\{\frac{z^k-\alpha^k}{\sqrt{2 \pi} k}\}_{k\in\mathbb{N}}$ is an orthonormal basis of $\Hsa^{1,2}(\mathbb{D})$. Thus, by the theory of reproducing kernel Hilbert spaces, $K_\alpha$ as given in Equation \eqref{ReproducingKernel} is the reproducing kernel for $\Hsa^{1,2}(\mathbb{D})$.

For the proof of part \textit{2.}, note that the reproducing kernel for $\mathbb{D}$ satisfies
\begin{align*}
  \overline{(K_0^z)'(w)}
   =\frac{\partial}{\partial\overline w}\sum_{k=1}^{\infty}\frac{z^k{\overline w}^k}{2\pi k^2}
    =\sum_{k=1}^{\infty}\frac{z^k{\overline w}^{k-1}}{2\pi k}=\frac{1}{2\pi\overline w}\Log\frac{1}{1-z\overline w}, \quad{w\in\mathbb{D}}.
\end{align*}
Hence for every $\zeta\in\bndry\mathbb{D}$, and since $T(\zeta) = i\zeta$, we have
\begin{align*}
\overline{T(\zeta)\dot{(K_0^z)'}(\zeta)} \quad = \quad \frac{1}{2\pi i}\Log\frac{1}{1-z\overline \zeta}, \quad{\zeta\in\bndry\mathbb{D}}.
\end{align*}
So for $g\in \neu^2(\bndry\mathbb{D})$ we have that Equation \eqref{dr} follows from the above and Theorem \ref{HilbertSpaceProp} part \textit{(\ref{HS3})}. 

For $g\in  \neu^{p}(b\mathbb D)$, $1\le p\le \infty$,  define $G$ as in \eqref{dr}. Then $G\in \vartheta(\mathbb D)$ and
$$G'(z) = \frac{1}{2\pi}\int\limits_{b\mathbb D}\frac{g(\zeta) \bar\zeta}{1-z\overline \zeta}\,d\sigma(\zeta) = \frac{1}{2\pi}\int\limits_{b\mathbb D}\frac{g(\zeta) }{\zeta-z }\,d\sigma(\zeta)= \frac{1}{2\pi i} \int\limits_{b\mathbb D}\frac{i\overline{ T(\zeta)} g(\zeta) }{\zeta-z }\,d \zeta  =  \mathbf C_{\mathbb D} (i\overline{T}g)(z),\ \ z\in \mathbb D. $$
Here we used the facts that $\zeta\bar\zeta=1$ and $d\sigma(\zeta) = \overline{ T(\zeta) } d\zeta $ on $\mathbb D$. Consequently, $(G')^*\in L^p(\bndry\mathbb D, \sigma)$ by the mapping property of the Cauchy integral $\mathbf C_{\mathbb D}$ and Cauchy transform  $\mathcal C_{\mathbb D}$. Moreover, from the above we also have that 
$$\dot{(G')}(\zeta)=  \mathcal C_{\mathbb D} (i\overline{T}g)(\zeta), \quad \mbox{a.e. } \zeta\in\bndry\mathbb{D}.$$
But $\overline{T}g\in h^p(b\mathbb D)$ because $g\in\neu^p(\bndry\mathbb{D})$, and $\mathcal C_{\mathbb D}$ is the identity on $h^p(b\mathbb D) $, thus 
$$\frac{\partial G}{\partial n}(\zeta)   =-iT(\zeta) \dot{(G')}(\zeta) = -iT(\zeta)\mathcal C_{\mathbb D}(i\overline{T}g )(\zeta) =  g(\zeta), \ \  \zeta\in \bndry\mathbb D\ \  \sigma-a.e..$$
That is, $G$ solves \eqref{NdbarIntro} for $1\le p\le \infty$. (Uniqueness was proved in \cite{GLZ}.)
\end{proof}

\medskip

\subsection{$\Hsa^{1,p}(\domain)$ is embedded in the Dirichlet Space}
In \cite{AxlShi}, Axler and Shields introduced the \textbf{Dirichlet space} $\mathcal{D}^2_{\alpha}(\domain)$ for
 a general domain $\domain$, namely
\[\mathcal{D}^2_{\alpha}(\domain) :=\left\{F\in\vartheta(\domain):F(\alpha)=0,\, \int\limits_{\domain}\abs{F'}^2 (z)\,dV(z)<\infty\right\},\quad \alpha\in \domain,\]
which is a Hilbert space with inner product
$$\langle F,G\rangle_{\mathcal{D}^2_{\alpha}(\domain)}:= \int\limits_{\domain} F'(z)\,\overline{G'(z)}\, dV(z).$$ 
(Here $dV$ is the Lebesgue measure for $\mathbb C$.) The analogous definition 
 of $\mathcal{D}^p_{\alpha}(\domain)$ with $1\leq p\leq\infty$ yields a Banach space with norm
$$
\| F\|_{\mathcal{D}^p_{\alpha}(\domain)} := \int\limits_\domain |F'(z)|^p\, dV (z).
$$ 
\begin{thm}\label{DirichletProp}
Let  $\domain$ be a bounded simply connected Lipschitz domain and $1<p<\infty$. Suppose that $F\in \Hsa^{1,p}(\domain)$. Then $F\in \mathcal{D}^p_{\alpha}(\domain)$ and

\[
\| F\|_{\mathcal{D}^p_{\alpha}(\domain)}\lesssim \|F\|_{\Hsa^{1,p}(\domain)}
.\]
That is, the holomorphic Sobolev-Hardy space is embedded in the Dirichlet space.
\end{thm}

To prove Theorem \ref{DirichletProp} we need the following result: 
\begin{lem}\label{T:Ddbar-C2}
Let  $\domain$ be a bounded simply connected Lipschitz domain and $1< p<\infty$. Suppose that $F\in \Hs^p(\domain)$. Then $F\in \vartheta(\domain)\cap L^p(\domain)$ and
$$ \|F\|_{L^p(\domain)}\lesssim \|\dot F\|_{L^p(\bndry\domain, \sigma)}.  $$
That is, the holomorphic Hardy space is embedded in the Bergman space.
\end{lem}
\begin{proof}
In \cite[Lemma 2.8]{GLZ} it is shown that if $1<p<\infty$ and $\domain$ is a simply connected and bounded Lipschitz domain, then for $F\in\Hs^p(\domain)$ quantities $\|F^*\|_{L^p(\bndry\domain, \sigma)}$  and $ \|\dot F\|_{L^p(\bndry\domain, \sigma)}$ are comparable. Thus it suffices to show that $ \|F\|_{L^p(\domain)}\lesssim \|F^*\|_{L^p(\bndry\domain, \sigma)}$. 

Consider a Ne\v{c}as exhaustion $\{\domain_k\}$ of $\domain$. Then there are finitely many coordinate rectangles $R_j:=[a_j,b_j]\times (c_j,d_j)$ with Lipschitz functions $\phi^j_k$ and $\phi^j$ whose graphs determine $\domain_k$ and $\domain$, respectively, on $R_j$ and $\phi^j_k$ converges uniformly to $\phi^j$. For any $k\in\mathbb{N}$, $x\in[a_j,b_j]$ and $y\in \left(\phi^j(x),\phi^j_k(x)\right]$, the point $x+iy$ lies directly above $x+i\phi^j(x)$ and thus $x+iy\in \Delta_1 + (x+i\phi^j(x))$, where $\Delta_1$ is the cone in Definition \ref{D:ap}. And so $e^{-i\theta_j}(x+iy)$ lies in $e^{-i\theta_j}(\Delta_1+ (x+i\phi^j(x)))\subseteq\Gamma(e^{-i\theta_j}(x+i\phi^j(x)))$. Fix $k\in\mathbb{N}$ so that for each $j$ we have $\|\phi^j_k-\phi^j\|_{\infty}<1$. Then we have for $F\in\Hs^p(\domain)$
\begin{eqnarray*}
\iint_{\domain-\domain_k} |F(z)|^p dA(z) &\leq & \sum_{j} \iint_{e^{-i\theta_j}R_j\cap(\domain-\domain_k)} |F(z)|^p dA(z) \\
&=& \sum_{j} \int\limits_{a_j}^{b_j} \int\limits_{\phi^j(x)}^{\phi_k^j(x)} |F(e^{-i\theta_j}(x+iy))|^p dy dx\\
&\leq & \sum_{j} \int\limits_{a_j}^{b_j} \int\limits_{\phi^j(x)}^{\phi_k^j(x)} F^*(e^{-i\theta_j}(x+i\phi_{j}(x)))^p dy dx\\
&\leq& \sum_{j} \int\limits_{a_j}^{b_j} F^*(e^{-i\theta_j}(x+i\phi_{j}(x)))^p dx\\
&\leq& \sum_{j} \int\limits_{a_j}^{b_j} F^*(e^{-i\theta_j}(x+i\phi_{j}(x)))^p \left|e^{i\theta_j}(1+i\phi_j'(x))\right| dx \\
&=& \sum_{j} \int\limits_{\bndry\domain\cap e^{-i\theta_j}R_j} F^*(\zeta)^p d\sigma(\zeta) \lesssim \|F^*\|_{L^p(\bndry\domain,\sigma)}^p.
\end{eqnarray*}
Since $\domain_k$ is compactly contained in $\domain$ and $k$ is fixed, $\dist(\domain_k, \bndry\domain)>d $ for some constant $d$ depending on $\domain$. So, similar to the argument of the proof of part \textit{(\ref{HS2})} of Theorem \ref{HilbertSpaceProp},  by the Cauchy integral formula we have $$\iint_{\domain_k} |F(z)|^p dA(z) \lesssim \|F^*\|_{L^p(\bndry\domain,\sigma)}^p,$$ completing the proof to  $\|F\|_{L^p(\domain)}\lesssim \|F^*\|_{L^p(\bndry\domain, \sigma)}$.
\end{proof}
\begin{proof}[Proof of Theorem \ref{DirichletProp}]
Let $F\in\Hsa^{1,p}(\domain)$. Then $F(\alpha)=0$ and $F'\in\Hs^p(\domain)$. By Lemma \ref{T:Ddbar-C2}, we also have $F'\in\vartheta(\domain)\cap L^p(\domain)$ giving that $F\in\mathcal{D}^p_\alpha(\domain)$, as desired.
\end{proof}

\medskip
%%%%%%%%%%%%%%%%%%%%%%%%%%%%%%%%%%%%%%%%%%%%%%%%%%%%%%%%%%%%%%%%%%%%%%%%%%%%%%%%%%%%%%%%%%%%%%%%%%%%%%%%%%%%%%%%%%%%%%%%%%%%%%
\section{Characterizations of $\neu^{p}(\bndry\domain)$ for simply connected $\domain$}\label{section3}
\begin{thm}\label{h^{1,p}Remarka}
Let $\domain$ be a bounded simply connected Lipschitz domain and $1\leq p\leq \infty$. Then  $\neu^{p}(\bndry\domain)$ defined as in \eqref{E:DefNeumannData} is closed in the $L^p(b\domain, \sigma)$-norm. Moreover, for
\begin{eqnarray*}
\neu_1&:=& \{Tg:g\in \ha^p(\bndry\domain)\},\\
\neu_2&:=& \left\{f\in L^p(\bndry\domain, \sigma): \int\limits_{\bndry\domain} \zeta^k f(\zeta) d\sigma(\zeta) = 0 \mbox{ for all $k=0,1,2,\ldots$}\right\},\\
\neu_3 &:=& \left\{f\in L^p(\bndry\domain,\sigma): \mathbf C_{\overline{D}^c} (\overline{T}f)= 0\right\}
\end{eqnarray*}
we have that $\neu^{p}(\bndry\domain)=\neu_1=\neu_2$. If $1<p<\infty$, then we also have $\neu^{p}(\bndry\domain)=\neu_3$.
\end{thm}
\begin{proof}
The inclusion $\neu^{p}(\bndry\domain)\subseteq \neu_1$ is immediate from \eqref{E:DefNeumannData}. 
 The reverse inclusion holds because $\domain$ is   simply connected and thus all holomorphic functions on $\domain$ have antiderivatives. As $\ha^p(\bndry\domain)$ is closed in the $L^p(b\domain, \sigma)$-norm, we see that $\neu_1$, and thus $\neu^{p}(\bndry\domain)$ is also closed.
 Next, the identity $\neu^{p}(\bndry\domain)=\neu_2$ follows from the fact that $T(\zeta)d\sigma(\zeta)=d\zeta$ and the well-known result of Smirnov that $g\in L^p(\bndry\domain,\sigma)$ lies in $\ha^p(\bndry\domain) $ if and only if $$\int\limits_{\bndry\domain} \zeta^k g(\zeta) d\zeta = 0 \ \ \text{ for}\ \  k=0,1,2,\ldots$$ See, for example, \cite[Theorem 10.4]{Duren}. Finally, the identity $\neu^{p}(\bndry\domain)=\neu_3$ for $1<p<\infty$ follows from Lemma \ref{hpCharacterization}.
\end{proof}

We may also characterize the elements of $\neu^p(\bndry \domain)$ for a bounded simply connected Lipschitz domain $\domain$ via its Riemann maps. We shall need the following  description of the tangent vector.
\begin{lem}\label{tangentpsi}
Let $\domain$ be a bounded simply connected Lipschitz domain and $\psi:\domain\to\mathbb{D}$ be a conformal map. Then the tangent vector $T$ of $\bndry\domain$ (which is defined a.e.) can be written as
\[T = i\frac{\dot{(\psi')}}{|\dot{(\psi')}|}\dot\psi,  \ \ \  \ \sigma\text{-a.e. on}\ \   \bndry\domain.\]
\end{lem}
\begin{proof}
Let $\phi:\mathbb{D}\to\domain$ be defined as $\phi=\psi^{-1}$. Since $\bndry\domain$ is Lipschitz, it is a Jordan curve so by Carath\'{e}odory's theorem $\phi$ extends to a homeomorphism of $\overline{\mathbb{D}}$ onto $\overline{\domain}$. By \cite[Theorem~3.13]{Duren}, we have that $\phi'\in\Hs^1(\mathbb{D})$ so that $\dot{(\phi')}$ exists $\sigma$-a.e.,  $\phi$ is absolutely continuous on $\bndry\mathbb{D}$, and
\begin{equation}\label{boundaryDerivative}
\frac{d}{dt}\phi(e^{it})=ie^{it}\dot{(\phi')}(e^{it}).
\end{equation}
Thus we can write the unit tangent vector $T$ via $\dot{(\phi')}$ for almost all $\zeta\in\partial\domain$. To do so, first note that for $r<1$
\[\psi'(\phi(re^{it}))=\frac{1}{\phi'(re^{it})}.\]
Since $\phi$ is conformal and $\dot{(\phi')}$ exists and is nonzero a.e., we see that the nontangential limit $\dot{(\psi')}$ exists a.e. and satisfies
\begin{equation}\label{nontangentialpsiprime}
\dot{(\psi')}(\zeta)=\frac{1}{\dot{(\phi')}(\psi(\zeta))}.
\end{equation}
Choose $t_0$ so that $\zeta=\phi(e^{it_0})$. Then by Equations \eqref{boundaryDerivative} and \eqref{nontangentialpsiprime} we have
\begin{equation*}
T(\zeta)=\left.\frac{\frac{d}{dt}\phi(e^{it})}{|\frac{d}{dt}\phi(e^{it})|}\right|_{t=t_0}=\frac{\dot{(\phi')}(e^{it_0})ie^{it_0}}{|\dot{(\phi')}(e^{it_0})|}=i\frac{\dot{(\phi')}(\psi(\zeta))\psi(\zeta)}{|\dot{(\phi')}(\psi(\zeta))|}=i\frac{|\dot{(\psi')}(\zeta)|\psi(\zeta)}{\dot{(\psi')}(\zeta)}\mbox{$\sigma$-a.e.},
\end{equation*}
as desired.
\end{proof}

\begin{thm}\label{PropVanish}
Let $\domain$ is a bounded simply connected Lipschitz domain, and $1\leq p\leq \infty$.   
Let $\psi:\domain\to\mathbb{D}$ be a conformal map with $\alpha: =\psi^{-1}(0)\in \domain$. Then
$$\neu^p(\bndry\domain) = \left\{ \frac{\dot{(\psi')}}{|\dot{(\psi')}|} \dot F: F\in  \mathcal H^p(\domain), F(\alpha)=0\right\}.  $$
\end{thm}
\begin{proof}  First by Proposition \ref{h^{1,p}Remarka}, one has
$$\neu^p(\bndry\domain) = \left\{T\dot G: G\in  \mathcal H^p(D) \right\}.  $$
Making use of  Lemma \ref{tangentpsi}, we further obtain
$$\neu^p(\bndry\domain) = \left\{\frac{\dot{(\psi')}}{|\dot{(\psi')}|} \dot\psi \dot G: G\in  \mathcal H^p(D) \right\}.  $$
Note that  $\psi$ is conformal on $\domain$ and continuous on $\overline{\domain}$. In particular, $\psi$ has only one zero at $\alpha$ and that zero is simple. Letting $F: =  \psi   G $, then$$   G \in \mathcal H^p(D) \ \ \text{ if and only if  }\ \    F \in \mathcal H^p(D),   F(\alpha)=0. $$ The proof is complete.
\end{proof}
Note that for $\domain=\mathbb{D}$ we can choose $\psi(z)=z$, in which case Theorem \ref{PropVanish} takes an especially simple form, namely
$$ \neu^p(\bndry\mathbb{D}) =  \{\dot F: F\in \mathcal H^p(\mathbb{D}), F(0)=0\}. $$
 %%%%%%%%%%%%%%%%%%%%%%%%%%%%%%%%%%%%%%%%%%%%%%%%%%%%%%%%%%%%%%%%%%%%%%%%%%%%%%%%%%%%%%%%%%%%%%%%%%%%%%%%%%%%%%%%%%%%%%
\section{A\! characterization\! of $\neu^p(\bndry\domain)$ for\! multiply\! connected\! $\domain$}\label{sectionmc}
Let $\domain$ be a  bounded Lipschitz domain. Then there exists $N\ge 1$, such that the boundary $\bndry\domain$ consists of $N$ closed rectifiable curves. Here and throughout we denote by $\gamma_1, \gamma_2,\ldots,\gamma_N$  those closed curves of $\bndry\domain$ endowed with the positive orientation, with $\gamma_N$ denoting the outer curve of $\bndry\domain$ (that is, $\domain$ lies in the set of points inside of $\gamma_N$).  

In order to characterize $\neu^p(\bndry\domain)$ we need to understand which elements of $\mathcal H^p(\domain)$ admit holomorphic antiderivatives. According to classical complex analysis theory, a continuous complex-valued function has an antiderivative in a domain $\domain$ (which may be simply or multiply-connected) if and only if the line integral of the function along every closed contour (i.e. piecewise $C^1$ path) in $\domain$ is zero. See,  for instance, \cite[Thereom 6.44]{ST}. 
This leads us to the following:

\begin{prop}\label{ii}Let $\domain$ be a bounded Lipschitz domain and let the boundary of $\domain$ be denoted as above. For $1\leq p\leq \infty$ and $F\in \Hs^p(\domain)$ we have that  $F$ is the complex derivative of a holomorphic function on $\domain$ if and only if \begin{equation}\label{i2}
    \int\limits_{\gamma_j} \dot F(\zeta) d\zeta =0\quad\mbox{
for all $\quad j = 1, \ldots, N$.}
\end{equation}
\end{prop}
\begin{proof} 
Let $\{\domain_k\}$ be Ne\v{c}as exhaustion of $\domain$ as defined in Lemma \ref{L:NecasExhaustion}. We will use the notation of Lemma \ref{L:NecasExhaustion} throughout this proof. For each $k$ and $1\leq j\leq N$, let $\gamma^k_j$ denote portion of $\bndry\domain_k$ such that $\Lambda_k(\gamma^k_j)=\gamma_j$. 
 
 First, assume  $F$ is a derivative of a holomorphic function on $\domain$. For each $k$ the curve $\gamma^k_j $ is a closed contour in $\domain$. Thus, by the Fundamental Theorem of Calculus, we have
$$  \int\limits_{\gamma^k_j} F(\zeta) d\zeta =0,\quad j=1, \ldots, N. $$
Thus
\begin{eqnarray}
0 &=&  \lim_{k\to\infty}\int\limits_{\gamma^k_j} F(\zeta) d\zeta = \lim_{k\to\infty}\int\limits_{\gamma^k_j} F(\eta_k)T_k(\eta_k)d\sigma_k(\eta_k)\nonumber\\
&=& \lim_{k\to\infty}\int\limits_{\gamma_j} F(\Lambda_k(\eta))T_k(\Lambda_k(\eta))w_k(\eta)d\sigma(\eta)\nonumber 
 =  \int\limits_{\gamma_j} \dot F(\eta)T(\eta)d\sigma(\eta) = \int\limits_{\gamma_j} \dot F(\zeta)d\zeta,\label{NecasLimit}
\end{eqnarray}
where we used the Dominated Convergence Theorem with the dominating function $M|F^*|$ (here we are using the fact that $F\in\Hs^p(\domain)$ so that $F^*\in L^1(\bndry\domain,\sigma)$), obtaining (\ref{i2}).

Conversely, assume (\ref{i2}) holds. Fixing a point $a\in \domain$, we shall show that for any $z\in \domain$, and any contour $\eta$  
in $\domain$ connecting $a$ and $z$,  the following line integral
$$ \int\limits_{\eta} F(\zeta) d\zeta$$
 is independent of the choice of the path. 

Indeed, let $\eta_1$ and $\eta_2$ be two contours joining $a$ and $z$ and let $\beta=\eta_1\cup(-\eta_2)$ be the closed contour starting and ending at $a$ (here $-\eta_2$ is $\eta_2$ oriented in the opposite direction). Without loss of generality, suppose $\beta$ is oriented counterclockwise and has no self-intersections. If the domain bounded by $\beta$ is a subset of $\domain$, then $\int_\beta F(\zeta)d\zeta =0$ by Cauchy's theorem. Else, for some $m$ between $1$ and $N$ there are $m$ components of $\bndry\domain$, say, $\gamma_{1},\ldots,\gamma_{m}$, that lie inside the domain bounded by $\beta$, while the remaining components $\gamma_{m+1},\ldots,\gamma_{N}$ lie outside of such domain.
With same notation as before, for a Ne\v{c}as exhaustion $\{\domain_k\}$, we choose $k$ large enough so that $\domain_k$ contains $\beta$, $\gamma^k_{1},\ldots,\gamma^k_{m}$ lie inside of $\beta$, and $\gamma^k_{m+1},\ldots,\gamma^k_{N}$ lie outside of $\beta$. By a generalized version of Cauchy's theorem (see, for example, \cite[Thereom 8.9]{ST}), 
$$  \int_\beta F(\zeta)d\zeta =  \sum_{\ell=1}^m\int_{\gamma^k_{\ell}} F(\zeta)d\zeta \quad \mbox{ for any large $k$.}$$
By an argument similar to the proof of Equation \eqref{NecasLimit} we have
$$  \int_\beta F(\zeta)d\zeta =  \lim_{k\to\infty} \sum_{\ell=1}^m\int_{\gamma^k_{\ell}} F(\zeta)d\zeta= \sum_{\ell=1}^m\int_{\gamma_{\ell}} \dot F(\zeta)d\zeta=0,$$
where we used \eqref{i2} in the last equality. 
Equivalently,  $$ \int\limits_{\eta_1} F(\zeta) d\zeta = \int\limits_{\eta_2} F(\zeta) d\zeta,$$ thus $$H(z): = \int\limits_{\eta} F(\zeta) d\zeta$$ is well defined and is a holomorphic antiderivative of $F$ on $\domain$.
\end{proof}
\begin{rem}\label{CauchyTheoremRemark} By Cauchy's theorem (in Lemma \ref{L:CauhyThmHp}), we have
\begin{equation*}
\sum_{j=1}^N \int\limits_{\gamma_j} \dot F(\zeta) d\zeta= \int\limits_{\bndry\domain} \dot F(\zeta) d\zeta =0,
\end{equation*}
  for any $F\in \Hs^p(\domain)$. Then we can refine the statement of Proposition \ref{ii} by requiring that only $(N-1)$-many terms in Equation \eqref{i2} vanish. Without loss of generality, we choose the first $(N-1)$ terms. Hence, Equation \eqref{i2} is equivalent to
\begin{equation}\label{i2v2}
    \int\limits_{\gamma_j} \dot F(\zeta) d\zeta =0\quad\mbox{
for all $\quad j = 1, \ldots, N-1$.}
\end{equation}
 
\end{rem}

\begin{thm}\label{Nm:SHsp-bd}
Let $\domain$ be a bounded Lipschitz domain and $1\leq p\leq \infty$.  
Then with $\neu^{p}(\bndry\domain)$ as in \eqref{E:DefNeumannData} we have
\begin{equation}\label{E:char}
\neu^{p}(\bndry\domain)= \left\{ Tf:\   f\in h^p(\bndry \domain),\quad  \int\limits_{\gamma_j}\!\! f(\zeta)\,d\zeta = 0, \quad
 1\leq j\leq N-1\right\}.
\end{equation}
\end{thm}

If $\domain$ is simply connected then the above identity reads $\neu^{p}(\bndry\domain)= \neu_1$, see Theorem \ref{h^{1,p}Remarka} (we should perhaps point out that the congruence of $\neu^p(\bndry\domain)$ with the two spaces $\neu_2$ and $\neu_3$ proved therein relies upon results that are classically stated for simply connected $\domain$).
\begin{proof}
Let
$$
L_0^p(\bndry\domain, \sigma): =\left\{g \in L^p(\bndry\domain, \sigma):\int\limits_{\bndry\domain}
\!\!
g(\zeta) d\sigma(\zeta) =0\right\}
$$ 
and
\[L_{00}^p(\bndry\domain, \sigma): =\left\{g\in L^p(\bndry\domain,\sigma): \ \int\limits_{\gamma_j}
\!\!
g(\zeta) d\sigma(\zeta) =0, \quad1\le j\le N\right\}.\]
 Obviously $L_{00}^p(\bndry\domain, \sigma)\subset L_{0}^p(\bndry\domain, \sigma)$. We claim that
 \begin{equation}\label{E:eq-1}
 \neu^{p}(\bndry\domain)=
 \{ g\in L_{00}^p(\bndry\domain, \sigma):  \overline{T} g\in h^p(\bndry \domain) \}.
 \end{equation}
Indeed, if $g\in \neu^{p}(\bndry\domain)$ there exists a $G\in\vartheta(\domain)$ with $G'\in \Hs^p(\domain)$ such that $g= -iT\dot{( G')}$, see \eqref{E:DefNeumannData}; hence $\bar{T} g = -i\dot{( G')}\in h^p(\bndry\domain)$. Moreover Proposition \ref{ii} gives that 
\[\int\limits_{\gamma_j} g(\zeta) d\sigma(\zeta)= -i\int\limits_{\gamma_j} \dot {(G')}(\zeta)d\zeta=0,\quad j=1,\ldots, N-1\]
proving that $g\in {L}_{00}^p(\bndry\domain,\sigma)$ and concluding the proof of the forward inclusion.
  For the reverse inclusion, suppose $g\in L_{00}^p(\bndry\domain, \sigma)$ and $g=T\dot F$ for some $F\in \Hs^p(\domain)$. Then 
\[\int\limits_{\gamma_j}  \dot F(\zeta)d\zeta =  \int\limits_{\gamma_j} g(\zeta)d\sigma(\zeta)=0,\quad j=1,\ldots, N-1\]
and it follows from Proposition \ref{ii} and Equation \eqref{i2v2} that $F$ has an antiderivative  $G\in \vartheta(\domain)$. Note that $iG\in \Hs^{1, p}(\domain)$ by definition. Thus, $g=T\dot F=-iT(\dot {iG')}\in \neu^{p}(\bndry\domain)$. The proof of \eqref{E:eq-1} is concluded. Equation \eqref{E:char} now follows since for $g$ as above we have $g=Tf$ with $f:=\overline{T} g$. 
\end{proof}

\fontsize{11}{11}\selectfont

\vspace{0.7cm}

\noindent williamgryc@muhlenberg.edu,

\vspace{0.2 cm}

\noindent Department of Mathematics and Computer Science, Muhlenberg College, Allentown, PA, 18104, USA.\\

\noindent loredana.lanzani@gmail.com, 

\vspace{0.2 cm}

\noindent Department of Mathematics, Syracuse University, Syracuse, NY, 13244, USA.

\vspace{0.2 cm}

\noindent Department of Mathematics, University of Bologna, Italy.
\noindent \\

\noindent jue.xiong@colorado.edu,

\vspace{0.2 cm}

\noindent Department of Mathematics, University of Colorado, Boulder, CO, 80309, USA.
\noindent \\

\noindent zhangyu@pfw.edu,

\vspace{0.2 cm}

\noindent Department of Mathematical Sciences, Purdue University Fort Wayne, Fort Wayne, IN 46805-1499, USA.\\
\end{document}